\documentclass[12pt]{article}
\usepackage{amsmath, amsthm, amscd, amsfonts, amssymb, graphicx, color}
\usepackage[bookmarksnumbered, colorlinks, plainpages]{hyperref}
\newtheorem{theorem}{Theorem}[section]

\theoremstyle{definition}
\newtheorem{definition}[theorem]{Definition}
\newtheorem{example}[theorem]{Example}

\theoremstyle{remark}
\newtheorem{remark}{Remark}[section]
\numberwithin{equation}{section}

\begin{document}
	
	\setcounter{page}{1}
	
	
	\begin{center}
		{\Large \textbf{$w$-Invexity and Optimality Problem}}
		
		\bigskip

		\textbf{Musavvir Ali$^{a,*}$} and \textbf{Ehtesham Akhter$^b$}\\
		\textbf{}  \textbf{}\\

		{\small $^{a,~b}$  Department of Mathematics,\\ Aligarh Muslim University, Aligarh-202002, India}
	\end{center}
	\noindent
	\footnote{$^*$Corresponding author\\
		E-mail addresses: musavvir.alig@gmail.com (M. Ali),\\ ehteshamakhter111@gmail.com (E. Akhter). }
	\bigskip
	
	{\abstract
	We define $w$-invex set, $w$-preinvex, $w$-strictly preinvex, $w$-quasi preinvex, $w$-strictly quasi preinvex, $w$-semi-strictly  quasi preinvex, and $w$-pre pseudo-invex functions in this context. And these form a class of real functions, which is the generalisation of a family of preinvex functions. Here, we provide a thorough analysis of the core characteristics of these functions, along with numerous examples that help to illustrate the idea. Finally, $w$-quasi preinvex, $w$-strictly preinvex, and $w$-strictly quasi preinvex functions are used to analyse the optimisation problems.}\\
	\section{Introduction}
	Numerous parts of mathematical programming, such as the adequate optimality condition and the duality theorem, deeply rely on convexity. One of the generalised convex functions is an invex function, which was first defined by Hanson \cite{Hanson}. Numerous generalisations of this idea have been offered in the literature. Invex functions were transformed into $\rho-(\eta, \theta)$-invex functions by Zalmai \cite{Zalmai}. Semistrictly preinvex functions on $\mathbb{R}^n$ were first developed by Yang and Li \cite{Li}. In addition to expanding several findings on convexity and optimization developed for linear spaces leading to Riemannian manifolds, Rapcsak \cite{Rapcsak} also obtained a generalisation known as geodesic convexity. In Riemannian spaces, the problem of convex programming with duality conditions was established by Udriste \cite{Udriste} Pini \cite{Pini}  introduced the idea of invex functions there. A vector programming problem described on a differentiable manifold admitting the necessary and sufficient conditions of KKT(Karush-Kuhn-Tucker) type by Mititelu \cite{Mititelu}. Ferrara and Mititelu proposed the Mond-Weir form of duality for vector programming issues on differentiable manifolds \cite{Ferrara}.

On convex-like functions, the alternative theorem and the assumption for the semi-preinvex function as a convex-like function together used to develop Fritz-John condition for ICOP (inequality constrained optimization problem, \cite{Jeyakumar} ). Fritz-John condition is also introduced as an example of the pre-variational inequality problem and duality. Following Hanson and Mond \cite{Hanson M} and Rueda and Hanson \cite{RUEDA}, who established the relationships between the K.T. point of a function of Type I and Type II  and at least one limited mathematical programming issue (see (\cite{Hanson M},\cite{RUEDA}).
The link between $w$-invexity and some proven generalizations of convexity will be covered in this study. The introduction to generalized convexity and related subjects are covered in Section 2. It is demonstrated that there is a fairly natural process to follow for $w$-quasi-preinvexity, $w$-strictly quasi-invexity, and $w$-semi-strictly quasi-invexity of the real functions in section 3, where $w$-preinvex, $w$-quasi preinvex, $w$-strictly quasi invex, and $w$-semi-strictly quasi preinvex functions will be discussed. We have also shown that $w$-preinvex functions provide desired outcomes, and it follows that $w$-preinvex functions can be transformed easily to become regular invex functions.
	\section{Preliminaries}
	\begin{definition}(\cite{Weir}-\cite{Weir T})
		A non empty subset $X$ of $\mathbb{R}^n$ is known as invex if $\exists$ a function $\eta : \mathbb{R}^n \times  \mathbb{R}^n \rightarrow  \mathbb{R}^n$ such that $\forall$ $ z_1, z_2 \in X$, $\forall$ $\delta \in [0, 1]$, we obtain $z_2+ \delta \eta(z_1, z_2) \in X$.
	\end{definition}
\begin{definition}(\cite{Weir}-\cite{Weir T})
A non empty subset $X$ of $\mathbb{R}^n$. A function $ h : X \rightarrow \mathbb{R}$ is known as preinvex if $\exists$ $\eta : \mathbb{R}^n \times  \mathbb{R}^n \rightarrow  \mathbb{R}^n$,  a function, such that $\forall$ $ z_1, z_2 \in X$, $\forall$ $\delta \in [0, 1]$,we obtain $$ h(z_2+\delta \eta(z_1, z_2))\leq \delta h(z_1)+(1-\delta)h(z_2).$$	
\end{definition}
\begin{definition}\cite{p invexity}
A non empty subset $X$ of $\mathbb{R}^n$. A function $ h : X \rightarrow \mathbb{R}$ is known as prequasi invex if $\exists$  $\eta : \mathbb{R}^n \times  \mathbb{R}^n \rightarrow  \mathbb{R}^n$, a function, s.t., $\forall$ $ z_1, z_2 \in X$, $\forall$ $\delta \in [0, 1]$,we obtain $$ h(z_2+\delta \eta(z_1, z_2))\leq \max \{h(z_1), h(z_2)\}.$$		
\end{definition}
\begin{definition}\cite{Yang}
	A non empty subset $X$ of $\mathbb{R}^n$. A function $ h : X \rightarrow \mathbb{R}$ is known as strictly prequasi invex if $\exists$ a function $\eta : \mathbb{R}^n \times  \mathbb{R}^n \rightarrow  \mathbb{R}^n$ such that $\forall$ $ z_1, z_2 \in X$ $z_1\neq z_2,$ $\forall$ $\delta \in (0, 1)$,we obtain $$ h(z_2+\delta \eta(z_1, z_2))
	< \max \{h(z_1), h(z_2)\}.$$	
\end{definition}
\begin{definition}\cite{Yang}
	A non empty subset $X$ of $\mathbb{R}^n$. A function $ h : X \rightarrow \mathbb{R}$ is known as semi-strictly prequasi invex if $\exists$ a function $\eta : \mathbb{R}^n \times  \mathbb{R}^n \rightarrow  \mathbb{R}^n$ such that $\forall$ $ z_1, z_2 \in X$ $h(z_1)\neq h(z_2),$ $\forall$ $\delta \in (0, 1)$, we obtain $$ h(z_2+\delta \eta(z_1, z_2))
	< \max \{h(z_1), h(z_2)\}.$$	
\end{definition}
	\section{$w$-invexity}
	\begin{definition}
	A non-empty subset $X$ of $\mathbb{R}^n$ is said to be {\bf$w$-invex} if there exists map $ \eta : \mathbb{R}^n \times \mathbb{R}^n \rightarrow \mathbb{R}^n$ and $w : \mathbb{R}^n \rightarrow \mathbb{R}^n$ such that $$ w(z_2)+ \delta \eta(z_1,  w(z_2)) \in X, ~~\forall z_1, z_2 \in X ~~\mbox{and} ~~\delta \in [0, 1]. $$
	\end{definition}
\begin{remark}
	Every invex set is $w$-invex (as take $w$ is an identity map) however converse isn\textsc{\char13}t always true.
\end{remark}
\begin{example}
	 	Let $ \eta :  [0, \infty) \times [0, \infty)  \rightarrow R $ be given by: $$ \eta(z_1, z_2)=z_1(z_2-2),$$  and let the function $ w :  [0, \infty)  \rightarrow R $ be given by: $$ w(z_2)=z_2+2. $$ Then $X$ is not a invex set w.r. to map $\eta$ but it is $w$-invex set w.r.to the maps $\eta$ and $w$.
\end{example}
\begin{definition}
	
		Let $X$ be a $w$-invex set. A function $h :  X \rightarrow \mathbb{R}$ is said to be

	\noindent	(i) {\bf$w$-preinvex}  if $\forall$ points  $z_1, z_2  \in X $, and  $ 0 \leq \delta \leq 1$, such that 
		$$ h( w(z_2)+ \delta \eta(z_1,  w(z_2)))\leq \delta h(z_1)+(1-\delta)h(z_2).
		$$ 
		(ii)  {\bf $w$-strictly preinvex} if $\forall$ points  $z_1, z_2  \in X $, $z_1\neq z_2$ and  $ 0 < \delta < 1$, such that 	$$ h( w(z_2)+ \delta \eta(z_1,  w(z_2)))< \delta h(z_1)+(1-\delta)h(z_2).$$ 
		(iii)  {\bf $w$-prequasi-invex} if $\forall$ points  $z_1, z_2  \in X $,  and  $ 0 \leq \delta \leq 1$, such that 	$$ h( w(z_2)+ \delta \eta(z_1,  w(z_2)))\leq \mbox{max} \{ h(z_1), h(z_2)\}.$$ 
		(iv)  {\bf $w$-strictly prequasi-invex} if $\forall$ points  $z_1, z_2  \in X $, $z_1\neq z_2$ and  $ 0 < \delta < 1$, such that 	$$ h( w(z_2)+ \delta \eta(z_1,  w(z_2)))< \mbox{max} \{ h(z_1), h(z_2)\}.$$ 
		(v)  {\bf $w$-semi-strictly prequasi-invex} if $\forall$ points  $z_1, z_2  \in X $, $h(z_1)\neq h(z_2)$ and  $ 0 < \delta < 1$, such that 	$$ h( w(z_2)+ \delta \eta(z_1,  w(z_2)))< \mbox{max} \{ h(z_1), h(z_2)\}.$$ 
\end{definition}
\begin{remark}
	In general, every $w$-strictly preinvex function is a also $w$-preinvex function, however converse isn\textsc{\char13}t always true. e.g., (3.7).
\end{remark}
\begin{remark}
	In general, every $w$-preinvex function is a also $w$-pre quasi-preinvex function, however converse isn\textsc{\char13}t always true. e.g., (3.6).
\end{remark}
\begin{remark}
	In general, every preinvex function is a also $w$-preinvex function (as take $w$ is an identity map), however converse isn\textsc{\char13}t always true. e.g.,(3.7).
\end{remark}
	\begin{example}
	Consider $ h :  \mathbb{R}  \rightarrow \mathbb{R} , $ be defined as: $ h(z_1)=z_1+k , $ where $k$ is a constant.
		
		Let the function $ w : \mathbb{R}  \rightarrow \mathbb{R} $ be defined as: $ w(z_2)=z_2-7,   $ and  consider $ \eta : \mathbb{R} \times \mathbb{R} \rightarrow \mathbb{R} $ as :  $ \eta(z_1,z_2)=z_1-z_2-6$ for all $z_1,z_2 \in \mathbb{R}$.
		Then the function $h$ is $w$-preinvex function w.r.to  $\eta$ and $w$.
		
	\end{example}
	\begin{example}
	Consider $ h :  \mathbb{R}  \rightarrow \mathbb{R} , $ be defined as: $ h(z_1)=z_1+k , $ where $k$ is a constant.
	
	Consider the function $ w : \mathbb{R}  \rightarrow \mathbb{R} $ be defined as: $ w(z_2)=z_2+6,   $ and  consider $ \eta : \mathbb{R} \times \mathbb{R} \rightarrow \mathbb{R} $ be defined as :  $ \eta(z_1,z_2)=z_1-z_2+6$ for all $z_1,z_2 \in \mathbb{R}$.
	Then the function $h$ is $w$-preinvex function w.r.to  $\eta$ and $w$, but not $w$-strictly preinvex and preinvex function.
	
\end{example}
	\begin{example}
	Consider $ h :  \mathbb{R}  \rightarrow \mathbb{R} , $ be defined as: $ h(z_1)=z_1^5 $.	Consider the function $ w : \mathbb{R}  \rightarrow \mathbb{R} $ be defined as: $ w(z_2)=z_2-6,   $ and  consider $ \eta : \mathbb{R} \times \mathbb{R} \rightarrow \mathbb{R} $ be defined as :  $ \eta(z_1,z_2)=z_1-z_2-6$ for all $z_1,z_2 \in \mathbb{R}$.
	Then the function $h$ is $w$-prequasi preinvex function w.r.to  $\eta$ and $w$, but not $w$-preinvex function.
	
\end{example}
	\begin{example}
	Let $ h :  \mathbb{R}  \rightarrow \mathbb{R} , $ be defined as: $ h(z_1)=z_1+k , $ where $k$ is a constant.
	
	let the function $ w : \mathbb{R}  \rightarrow \mathbb{R} $ be defined as: $ w(z_2)=z_2+6,   $ and  consider $ \eta : \mathbb{R} \times \mathbb{R} \rightarrow \mathbb{R} $ be defined as :  $ \eta(z_1,z_2)=z_1-z_2+6$ for all $z_1,z_2 \in \mathbb{R}$.
	It is simple to demonstrate that while $h$ is $w$-preinvex w.r.to $\eta$ and $w$ and, it is not preinvex w.r.to $\eta$.

\end{example}
	\begin{example}
	Let $ h :  [0, \infty)  \rightarrow \mathbb{R} , $ be defined as: $ h(z_1)= \begin{cases}
		11, \text{ if } z_1 \in [0, 11) \\
		-11, \text{ if }  z_1 \in [11, \infty)
	\end{cases}  ,$ 	
	let the function $ w :  [0, \infty)   \rightarrow \mathbb{R} $ be defined as: $ w(z_2)=z_2+11,   $ and  consider $ \eta :  [0, \infty)  \times  [0, \infty)  \rightarrow \mathbb{R} $ be defined as :  $ \eta(z_1,z_2)=z_1^2+z^2_2+11$ for all $z_1,z_2 \in  [0, \infty) $.
	Then the function $h$ is $w$-pre quasi-invex function w.r.to $\eta$ and $w$.
	
\end{example}
\begin{theorem}
	  $ h : X \rightarrow \mathbb{R}$ is $w$-preinvex function w.r.to $ \eta$ and $w$ if and only if the epigraph of $h$, $$ epi(h)=\{(z, \alpha) : \alpha \in \mathbb{R}, h(z)\leq \alpha \},$$ is $w$-invex w.r.to the same $\eta$ and $w$,  where $X$ is any non-empty $w$-invex subset  of $\mathbb{R}^n$.
\end{theorem}
\begin{proof}
	Consider $ h : X \rightarrow \mathbb{R}$ be a $w$-preinvex function w.r.to $\eta$ and $w$. Suppose that $(z_1, \alpha) $ and $(z_2, \beta ) \in epi(h)$, with $ z_1, z_2 \in X$, $h(z_1)\leq \alpha$ and  $h(z_2)\leq \beta$, for $\alpha, \beta \in \mathbb{R} $. Thus,
	 \begin{eqnarray*}
	 		 h( w(z_2)+ \delta \eta(z_1,  w(z_2)))&\leq& \delta h(z_1)+(1-\delta)h(z_2).\notag\\
	 		 &=& \delta \alpha +(1-\delta)\beta, ~~\mbox{for}~ \delta \in [0, 1] \notag\\
	 \end{eqnarray*}
 which implies that
 $$( w(z_2)+ \delta \eta(z_1,  w(z_2)), ~\delta \alpha +(1-\delta)\beta) \in epi(h),$$
 showing that $epi(h)$ is $w$-invex w.r.to the maps $\eta$ and $w$.
 Conversely, assume that $epi(h)$ is $w$-invex set w.r.to the maps $\eta$ and $w$. Take $z_1, z_2 \in X$. Then $(z_1, h(z_1)), (z_2, h(z_2)) \in epi(h)$, it follows that
  $$( w(z_2)+ \delta \eta(z_1,  w(z_2)), ~\delta h(z_1) +(1-\delta)h(z_2)) \in epi(h), ~~\forall \delta \in [0, 1]$$ $\implies$  $$h( w(z_2)+ \delta \eta(z_1,  w(z_2)))\leq \delta h(z_1)+(1-\delta)h(z_2).$$ 
  By doing so, it is demonstrated that h is a w-preinvex function w.r.to the map $\eta$ and $w$. Completeness of the proof.
  
\end{proof}
\begin{theorem}
	Consider $h : X \rightarrow \mathbb{R}$ be a $w$-preinvex function w.r.to the functions $\eta$ and $w$,  where $X$ is any non-empty $w$-invex subset  of $R^n$. Then, the level set $$ M_{\alpha}=\{z \in X, h(z)\leq \alpha, \alpha \in \mathbb{R}\}$$ is $w$-invex w.r.to the same map $\eta$ and $w$.
\end{theorem}
\begin{proof}
Let $z_1, z_2 \in M_{\alpha}$. Then $z_1, z_2  \in X$ and $h(z_1)\leq \alpha, h(z_2)\leq \alpha $. Since $h$ is $w$-preinvex then, we have, $\forall \in [0, 1]$, 
 \begin{eqnarray*}
	h( w(z_2)+ \delta \eta(z_1,  w(z_2)))&\leq& \delta h(z_1)+(1-\delta)h(z_2).\notag\\
	&=& \delta \alpha +(1-\delta)\alpha, ~~\mbox{for}~ \delta \in [0, 1] \notag\\
		&=& \alpha. \notag\\
\end{eqnarray*}
This shows that $ w(z_2)+ \delta \eta(z_1,  w(z_2)) \in X$, and the level set $M_{\alpha}$ is a $w$-invex set w.r.to the map $\eta$ and $w$.
\end{proof}
\begin{theorem}
	The function $h: X \rightarrow \mathbb{R}$ is $w$-quasi preinvex w.r.to the same $\eta$ and $w$ $\iff$ the level set $$ M_{\alpha}=\{z \in X, h(z)\leq \alpha, \alpha \in \mathbb{R}\}$$ is $w$-invex w.r.to the same $\eta$ and $w$,  where $X$ is any non-empty $w$-invex subset  of $R^n$.
\end{theorem}
\begin{proof}
	Suppose that the function $h$ is $w$-quasi preinvex w.r.to the maps $\eta$ and $w$. Let $z_1,z_2 \in M_{\alpha}.$ Then,  $z_1,z_2 \in X$ and $\mbox{max} \{h(z_1),h(z_2)\} \leq \alpha$. Since $X$ is a $w$-invex set, so, $\forall ~\delta \in [0, 1],$ $$  ( w(z_2)+ \delta \eta(z_1,  w(z_2)) \in X.$$ Since $h$ is $w$-quasi preinvex function, we have 	$$ h( w(z_2)+ \delta \eta(z_1,  w(z_2)))\leq \mbox{max} \{ h(z_1), h(z_2)\}\leq \alpha,$$ $\implies$  $$   (w(z_2)+ \delta \eta(z_1,  w(z_2)) \in M_{\alpha}.$$ This implies that $M_{\alpha}$ is a $w$-invex set w.r.to the maps $\eta$ and $w$.\\
	
	Conversely assume that the level set $M_{\alpha}$ is $w$-invex w.r.to the maps $\eta$ and $w$. Let $z_1, z_2 \in M_{\alpha}$ for $\mbox{max}\{h(z_1), h(z_2)\}=\alpha.$ Since $X$ is a $w$-invex set, we have  $$w(z_2)+ \delta \eta(z_1,  w(z_2)) \in X,$$ implies that 	$$ h( w(z_2)+ \delta \eta(z_1,  w(z_2)))\leq \mbox{max} \{ h(z_1), h(z_2)\}= \alpha.$$ This prove that $h$ is a $w$-quasi preinvex function w.r.to the functions $\eta$ and $w$.
\end{proof}
\begin{theorem}
	Assume that $ h : X \rightarrow \mathbb{R}$ is a $w$-preinvex function w.r.to the maps $\eta$ and $w$,  where $X$ is any non-empty $w$-invex subset  of $R^n$ and $\nu= inf_{z\in X} h(z).$ Then, the set $$ F=\{z\in X : h(z)=\nu \}$$ is $w$-invex w.r.to the same $\eta$ and $w$. If $h$ is $w$-strictly preinvex then $F$ is a singleton.
\end{theorem}
\begin{proof}
	Assume that $z_1, z_2 \in X$ and $ \delta \in [0, 1]$. Since $h$ is a $w$-preinvex then, we have
	 \begin{eqnarray*}
		h( w(z_2)+ \delta \eta(z_1,  w(z_2)))&\leq& \delta h(z_1)+(1-\delta)h(z_2)\notag\\
		&=& \delta \nu +(1-\delta)\nu, ~~\mbox{for}~ \delta \in [0, 1] \notag\\
		&=& \nu. \notag\\
	\end{eqnarray*}
	This follows that $$ w(z_2)+\delta \eta(z_1, w(z_2)) \in F,$$ which implies that $F$ is $w$-invex w.r.to the maps $\eta$ and $w$.\\
	For the other part, contrarily suppose that $h(z_1)=h(z_2)=\nu$. Since $X$ is a $w$-invex set, so $$ w(z_2)+\delta \eta(z_1, w(z_2)) \in F,~~ \forall ~~\delta\in [0, 1].$$ Also, since $h$ is $w$-strictly preinvex, we have 
	 \begin{eqnarray*}
		h( w(z_2)+ \delta \eta(z_1,  w(z_2)))&<& \delta h(z_1)+(1-\delta)h(z_2)\notag\\
		&=& \delta \nu +(1-\delta)\nu, ~~\mbox{for}~ \delta \in [0, 1] \notag\\
		&=& \nu. \notag\\
	\end{eqnarray*}
This contradiction the fact that $\nu= inf_{z\in X} h(z).$ Therefore, $F$ is singleton.
\end{proof}
\begin{theorem}
	Let $ h : X \rightarrow \mathbb{R}$ is a $w$-preinvex function on $X$,  where $X$ is any non-empty $w$-invex subset  of $R^n$. Therefore, any local minimum point is also a global minimum point.

\end{theorem}
\begin{proof}Assuming that $z_2 \in X$ is not a global minimum point, let $z_2$ be the local minimum point for $h$. Then, $\exists$ $z_1 \in X$ such that, $ h(z_1)\leq h(z_2)$. Since $h$ is $w$-preinvex $X$, $\exists$ $\eta$ and $w$ such that, $\forall$ $\delta \in [0, 1]$, we have 
	 \begin{eqnarray}
		h( w(z_2)+ \delta \eta(z_1,  w(z_2)))&\leq& \delta h(z_1)+(1-\delta)h(z_2)\notag\\
		&<&\delta h(z_2)+(1-\delta)h(z_2) \notag\\
		&=&h(z_2).
	\end{eqnarray}
Given that  $z_2$ represents a global minimum then $\exists$ a neighbourhood $N$ of $z_2$ such that \\$ h(z)\geq h(z_2)$, $\forall$ $z \in N\cap X$. As $ w(z_2)+ \delta \eta(z_1,  w(z_2)) \in X$ for each $\delta \in [0, 1]$, $\exists$ $\epsilon >0$ such that $ w(z_2)+ \delta \eta(z_1,  w(z_2)) \in X\cap N$ for each $ \delta \in [0, \epsilon)$, in a contradiction with 3.1.
\end{proof}
\begin{theorem}
	Let $ h : X \rightarrow \mathbb{R}$ is a $w$-preinvex function on $X$, where $X$ is any non-empty $w$-invex subset  of $R^n$. Any strict local minimum point is thus a strict global minimum point. 
\end{theorem}
\begin{proof}
The proof is same lines as of Theorem 3.13.	
\end{proof}

\begin{theorem}
	If $ h : X \rightarrow \mathbb{R}$ is $w$-preinvex on the $w$-invex set $X$,  where $X$ is any non-empty $w$-invex subset  of $R^n$, then also $kh$ is $w$-preinvex w.r.to $\eta$ and $w$, for any $k>0$.
\end{theorem}
\begin{proof}
	Since $h$ is a $w$-preinvex function then, we have $\forall$ $z_1, z_2 \in X$ and  $\delta \in [0, 1]$ $$	h( w(z_2)+ \delta \eta(z_1,  w(z_2)))\leq \delta h(z_1)+(1-\delta)h(z_2),$$ we have, for any $k>0$, $$	k(h( w(z_2)+ \delta \eta(z_1,  w(z_2))))\leq k(\delta h(z_1)+(1-\delta)h(z_2)), ~~\forall \delta \in [0, 1]$$ from which  $$	k(h( w(z_2)+ \delta \eta(z_1,  w(z_2))))\leq \delta (kh)(z_1)+(1-\delta)(kh)(z_2), ~~ \forall \delta \in [0, 1].$$	
\end{proof}
\begin{theorem}
Let $ h_1, h_2 : X \rightarrow \mathbb{R}$ be two $w$-preinvex functions w.r.to the same functions $\eta$ and $w$,  where $X$ is any non-empty $w$-invex subset  of $R^n$. Then, $h_1+h_2$ is a $w$-preinvex function w.r.to same $\eta$ and $w$.
\end{theorem}
\begin{proof}
	Because $h_1$ and $h_2$ are $w$-preinvex w.r.to $\eta$ and $w$, we have $$	h_1( w(z_2)+ \delta \eta(z_1,  w(z_2)))\leq \delta h_1(z_1)+(1-\delta)h_1(z_2).$$ And $$	h_2( w(z_2)+ \delta \eta(z_1,  w(z_2)))\leq \delta h_2(z_1)+(1-\delta)h_2(z_2).$$
By combining the aforementioned relationships, we get
	$$(	h_1+h_2)( w(z_2)+ \delta \eta(z_1,  w(z_2)))\leq \delta (h_1+h_2)(z_1)+(1-\delta)(h_1+h_2)(z_2).$$
\end{proof}
Theorems 3.15 and 3.16 have a direct impact on the following.
\begin{theorem}
	Let $h_i : X \rightarrow \mathbb{R}$, $i=1,2,...,m,$ be $w$-preinvex w.r.to same $\eta$ and $w$,  where $X$ is any non-empty $w$-invex subset  of $R^n$. Then $\sum_{i=1}^{m} k_ih_i(z)$ is $w$-preinvex w.r.to same $\eta$ and $w$, where $k_i>0$, $i=1,2,...,m.$
\end{theorem}

\begin{theorem}
	Let $h: X \rightarrow \mathbb{R}$ be a $w$-preinvex w.r.to $\eta$ and $w$,  where $X$ is any non-empty $w$-invex subset  of $R^n$, let $\phi : \mathbb{R} \rightarrow \mathbb{R}$ be an increasing and convex. Then $\phi \circ h$ is $w$-preinvex w.r.to $\eta$ and $w$.
\end{theorem}
\begin{proof}
	As $h$ is $w$-preinvex w.r.to $\eta$ and $w$, we obtain $\forall$  $z_1, z_2 \in X, \delta \in [0, 1]$,  $$	h( w(z_2)+ \delta \eta(z_1,  w(z_2)))\leq \delta h(z_1)+(1-\delta)h(z_2).$$ Due to the fact that $\phi : \mathbb{R} \rightarrow \mathbb{R}$ non-decreasing and convex, we obtain 
	 \begin{eqnarray*}
		\phi (h( w(z_2)+ \delta \eta(z_1,  w(z_2))))&\leq& \phi (\delta  (h(z_1)+(1-\delta)h(z_2))\notag\\
		&\leq& \delta \phi (h(z_1)) +(1-\delta)\phi (h(z_1)),  \notag\\
	\end{eqnarray*}
or
$$\phi \circ h( w(z_2)+ \delta \eta(z_1,  w(z_2))\leq \delta \phi \circ h(z_1) +(1-\delta)\phi\circ h(z_1)$$
\end{proof}
\begin{definition}
Let $X$ is any non-empty $w$-invex subset  of $R^n$.  $h : X \rightarrow \mathbb{R}$ is said to be $w$-pre pseudo-invex function if $\exists$ functions $\eta$, $w$ and a  function $b$, positve on domain, such that $$ h(z_1)<h(z_2) \implies h(z_2+\delta \eta(z_1, w(z_2))) \leq h(z_2)+\delta (\delta-1)b(z_1, z_2)$$ for every $\delta \in (0, 1)$ and $ z_1, z_2 \in X$.
\end{definition}
\begin{theorem}
	 If $h$ is $w$-preinvex, then $h$ is $w$-pre pseudo-invex w.r.to the same $\eta$ and $w$, where $X$ is any non-empty $w$-invex subset  of $R^n$.
\end{theorem}
\begin{proof}
	For $ h(z_1)<h(z_2)$, for every $\delta \in (0, 1),$ we can write
	 \begin{eqnarray*}
		(h( w(z_2)+ \delta \eta(z_1,  w(z_2))))&\leq&   h(z_2)+\delta [h(z_1)-h(z_2)]\notag\\
		&<&  h(z_2)+\delta [h(z_1)-h(z_2)]-\delta^2 [h(z_1)-h(z_2)],  \notag\\
			&=&  h(z_2)+\delta(\delta-1) [h(z_2)-h(z_1)],  \notag\\
	\end{eqnarray*}
where $b(z_1,z_2)=h(z_2)-h(z_1)>0.$
\end{proof}
\begin{theorem}
		Let  $ h : X \rightarrow \mathbb{R}$ is a $w$-pre pseudo-invex function,  where $X$ is any non-empty $w$-invex subset  of $R^n$. Consequently, any local minimum point is also a global minimum point.
\end{theorem}
\begin{proof}Let us  assume irrationally that $z_2$ be local minimum point rather than a global one.
	 Then $\exists$ $z_1 \in X$ such that $h(z_1)<h(z_2).$ According to the $w$-pre pseudo-invexity definition, which means that, we have $\exists$ $ b(z_1, z_2)>0$, such that, $\forall$ $\delta \in (0, 1)$,  $$  h(z_2+\delta \eta(z_1, w(z_2))) \leq h(z_2)+\delta (\delta-1)b(z_1, z_2)<h(z_2),$$ from which $h(z_2)> h(w(z_2)+\delta \eta (z_1, w(z_2)))$, ~~$\forall$ $\in (0, 1),$\\ in contradiction with the assumption.
\end{proof}
\begin{theorem}
 Assume that $\eta(z_1, w(z_2))\neq 0$ whenever $z_1\neq w(z_2)$. Let $h$ be a $w$-pre quasi-invex function on $X$ w.r.to $\eta$ and $w$,  where $X$ is any non-empty $w$-invex subset  of $R^n$. Consequently, any strict local minimum point is $h$ also a strict global minimum point.
\end{theorem}
\begin{proof} Contrarily suppose that $z_2$ be strict local minimum but  not global; then $\exists$  $z \in X$ s.t., $h(z) < h(z_2).$ Since $h$ is $w$-pre quasi-invex, we have  $$  h(z_2+\delta \eta(z, w(z_2))) \leq h(z_2),$$ a contradiction.
\end{proof}
\begin{theorem}
	Let  $h : X \rightarrow \mathbb{R}$ be a $w$-pre quasi-invex function w.r.to $\eta$,  where $X$ is any non-empty $w$-invex subset  of $R^n$ and $w$. Suppose that $ \phi : \mathbb{R} \rightarrow \mathbb{R}$ is a increasing function, then, $\phi \circ h$ is $w$-pre quasi-invex w.r.to $\eta$ and $w$.
\end{theorem}
\begin{proof}
	Given that $h$ is $w$-pre quasi-invex function and $\phi$ is a increasing, we obtain 	$$\phi \circ h( w(z_2)+ \delta \eta(z_1,  w(z_2)))\leq \phi(\mbox{max} \{ h(z_1), h(z_2)\}).$$ 
	 \begin{eqnarray*}
		\phi \circ h( w(z_2)+ \delta \eta(z_1,  w(z_2)))&\leq& \phi(\mbox{max} \{ h(z_1), h(z_2)\})\notag\\
		&=& \mbox{max} \{\phi \circ h(z_1),\phi \circ h(z_2)\},
	\end{eqnarray*}
 which expresses the fact that the composition function $\phi \circ h$ is $w$-pre quasi-invex.

\end{proof}
\section{Optimality}
Let $ h : \mathbb{R}^n \rightarrow \mathbb{R}$ and $ g_i:\mathbb{R}^n \rightarrow \mathbb{R}$, $i=1,2,....,m $ are $w$-preinvex functions on $\mathbb{R}^n$. A $w$-preinvex programming problem is formulated as follows:

\begin{equation}
	\begin{aligned}
		& {\text{Min~h(z)}}\\
		& \text{subject to}\\
		& \text z \in X=\{z\in \mathbb{R}^n : g_i(z)\leq 0, \; i = 1, \ldots, m\}.
	\end{aligned}
\end{equation}
\begin{theorem}
	If $ h : \mathbb{R}^n \rightarrow \mathbb{R}$ is $w$-strictly preinvex function on $X$,  where $X$ is any non-empty $w$-invex subset  of $R^n$, then the global optimal solution of problem (4.1) is unique.
\end{theorem}
\begin{proof}
	Let $z_1,z_2 \in X$ be different global optimal solution of problem (4.1). Then $h(z_1)=h(z_2)$. Since $h$ is $w$-strictly invex then, we have 
	\begin{eqnarray*}
	 h( w(z_2)+ \delta \eta(z_1,  w(z_2)))&<& \delta h(z_1)+(1-\delta)h(z_2)\notag\\
		&=& h(z_1),
	\end{eqnarray*}
$$	h( w(z_2)+ \delta \eta(z_1,  w(z_2)))<  h(z_1), ~~for~~ each~~ \delta \in (0, 1).$$ This contradicts the optimality of $z_2$ for problem(4.1). Hence, the global optimal solution of problem (4.1) is unique.
\end{proof}
\begin{theorem}
	Let $ h : \mathbb{R}^n \rightarrow \mathbb{R}$ be $w$-quasi preinvex function on $X$,  where $X$ is any non-empty $w$-invex subset  of $R^n$ and let $\gamma=\min_{z\in X} h(z)$. The set $X=\{z\in X : h(z)=\gamma\}$ of optimal solution of problem (4.1) is $w$-invex set.\\
	If $h$ is $w$-strictly quasi-invex, then the set $X$ is a singleton.
\end{theorem}
\begin{proof}
		Let $z_1,z_2 \in X$ be different global optimal solution of problem (4.1). Then $$h(z_1)=h(z_2)=\gamma.$$ Since $ h : \mathbb{R}^n \rightarrow \mathbb{R}$ be $w$-quasi preinvex function on $X$, then 	$$ h( w(z_2)+ \delta \eta(z_1,  w(z_2)))\leq \mbox{max} \{ h(z_1), h(z_2)\}= \gamma$$ which implies that $h( w(z_2)+ \delta \eta(z_1,  w(z_2)) \in X,$ so $X$ is $w$-invex set.\\
		For the other part, on contrary let $z_1,z_2 \in X$, $z_1\neq z_2, ~~\delta \in (0, 1),$ then $h( w(z_2)+ \delta \eta(z_1,  w(z_2)) \in X.$\\Further, since $h$ is $w$-strictly quasi preinvexity function on $X$, we have 	$$ h( w(z_2)+ \delta \eta(z_1,  w(z_2)))\leq \mbox{max} \{ h(z_1), h(z_2)\}= \gamma.$$ This contradicts that $\gamma=\min_{z\in X} h(z)$ and hence the results follows.
\end{proof}
\begin{theorem}
	Consider the function $ h : \mathbb{R}^n \rightarrow \mathbb{R}$ is $w$-pre quasi-preinvex on a $w$-invex subset of $\mathbb{R}^n$. Suppose $\gamma=\min_{z \in X} h(z)$. Then the optimal solutions of problem (4.1) i.e., $Y=\{z \in X\ : h(z)=\gamma\}$ is a $w$-invex set. If $h$ is $w$-strictly quasi-preinvex on a $w$-invex subset of $\mathbb{R}^n$. Then $w$-invex set a singleton set.
\end{theorem}
\begin{proof}
Let $z_1, z_2 \in Y$ and $\delta \in [0, 1]	$, then $z_1, z_2 \in X$ also $h(z_1)=\gamma$, $h(z_2)=\gamma$.
As $h$ is $w$-pre quasi preinvex function on $w$-invex set, we have $$ w(z_2)+\delta \eta(z_2, w(z_1)) \in X$$ and $$  h(w(z_2)+\delta \eta(z_2, w(z_1)))\leq \max \{h(z_1), h(z_2)\}=\gamma$$ which shows that $$ (w(z_2)+\delta \eta(z_2, w(z_1)) \in Y,$$ it demonstrate that $X$-invex set.\\
For the other part, suppose to the contrary that $z_1, z_2 \in X$, and $z_1\neq z_2$, $\delta \in (0, 1)$, we get $$ w(z_2)+\delta \eta(z_2, w(z_1)) \in X$$ As $h$ is $w$-strictly quasi preinvex on $X$, we get $$  h(w(z_2)+\delta \eta(z_2, w(z_1)))< \max \{h(z_1), h(z_2)\}=\gamma,$$ there is a contradiction.
\end{proof}

\end{document}